\documentclass{daj}

\usepackage{amsmath,amsthm,amsfonts,bbm} % bbm necessary to get the ``blackboard 1'' to work

\dajAUTHORdetails{%
  title = {Kneser Graphs are like Swiss Cheese}, %% please capitalize all significant words
  author = {Ehud Friedgut and Oded Regev},
  plaintextauthor = {Ehud Friedgut, Oded Regev},
  keywords = {product graphs, Kneser graph, independent set},
}   %%% END \dajAUTHORdetails

\dajEDITORdetails{%
   year={2018},
   number={2},
   received={15 February 2017},   % received date: example: 7 January 2017
   published={23 January 2018},  % published date
   doi={10.19086/da3103},       % XXX = number of paper, e.g. da006 for paper#6
}   %%% END \dajEDITORdetails

\newcommand{\R}{\ensuremath{\mathbb{R}}}
\newcommand{\N}{\ensuremath{\mathbb{N}}}
\newcommand{\eps}{\varepsilon} 
\newcommand{\ip}[1]{\langle{#1}\rangle} % inner product
\newcommand{\indic}{\mathbbm{1}} 
\DeclareMathOperator{\E}{\mathbb{E}}
\newcommand{\tower}{\mbox{\rm Tower}}
\newcommand{\edge}{\mbox{\rm Edge}}

\newtheorem{theorem}{Theorem}[section]
\newtheorem{thm}{Theorem}
\newtheorem{cor}[theorem]{Corollary}
\newtheorem{definition}[theorem]{Definition}
\newtheorem{claim}[theorem]{Claim}
\newtheorem{lemma}[theorem]{Lemma}

\begin{document}

\begin{frontmatter}[classification=text]
%% EDITOR: this will force the keywords to appear right after the Abstract.
%%   If the abstract is too long and would force the keywords off the
%%   front page, please comment out % [classification=text] above
%%   This way the keywords will be floated on the bottom of the first page
%%   even though the Abstract spills over to the next page.

%%% AUTHOR: Title goes here.  This line is optional.  You must use it
%%   if title has footnote attached or requires nontrivial typesetting,
%%   e.g., inclusion of linebreaks to force nice layout.
%\title{Kneser Graphs are like Swiss Cheese\footnote{This is a footnote to the title}} %% please capitalize all significant words

%%% AUTHOR:
%%% List all authors. If you wish, place grant acknowledgements in \thanks.
%%% In brackets include a short tag for each author.
\author[friedgut]{Ehud Friedgut\thanks{Research supported in part by ISF grant 1168/15 and by Minerva grant 712023.}}
\author[regev]{Oded Regev\thanks{This material is based upon work supported by the National Science Foundation under Grant No.~CCF-1320188. Any opinions, findings, and conclusions or recommendations expressed in this material are those of the authors and do not necessarily reflect the views of the National Science Foundation.}}

%%% AUTHOR: Abstract goes here
\begin{abstract}
We prove that for a large family of product graphs, and for Kneser graphs $K(n,\alpha n)$ with fixed $\alpha <1/2$, the following holds. Any set of vertices that spans a small proportion of the edges in the graph can be made independent by removing a small proportion of the vertices of the graph. This allows us to strengthen the results of \cite{DinurFR06} and \cite{DinurF09}, and show that any independent set in these graphs is almost contained in an independent set which depends on few coordinates. Our proof is inspired by, and follows some of the main ideas of, Fox's proof of the graph removal lemma \cite{Fox11}.
\end{abstract}
\end{frontmatter}

\section{Introduction}
The celebrated triangle removal lemma of Ruzsa and Szemer\'edi  \cite{RS} has a deceptively simple formulation, yet is in fact a deep structural result. It is known to imply Roth's theorem on three term arithmetic progressions, and its generalizations to hypergraphs imply Szemer{\'e}di's theorem on arithmetic progressions. The statement is as follows. 
\begin{thm}[Ruzsa-Szemer{\'e}di]
For every $\eps>0$ there exists a $\delta > 0$ such that if a graph on $n$ vertices spans fewer than $\delta n^3$ triangles, it can be made triangle free by removing at most $\eps n^2$ edges. 
\end{thm}
This statement can be formulated as a statement regarding a 3-uniform hypergraph whose vertices are the edges of the complete graph on $n$ vertices, and whose edges are the triangles. The statement then says that every set of vertices that spans few edges can be made independent by removing a small number of vertices.

In this paper we take this statement ``one level down'' to graphs. We show that a large family of graphs has an ``edge removal phenomenon,''
i.e., any set of vertices that spans $o(|E|)$ edges can be made independent by removing $o(|V|)$ vertices.
In particular, we shall prove this for product graphs (see Theorem~\ref{thm:main} for the precise statement), 
a good example of which is $K_3^n$, in which the set of vertices is $\{0,1,2\}^n$, and two vertices span an edge if they differ in all coordinates. (Note that this is very different from the Hamming graph, where neighbors differ in precisely one coordinate.) 
We will also prove this for Kneser graphs $K(n, k)$ (see Theorem~\ref{thm:Kneser}), where the vertices are the subsets of size $k$ of $[n]$ for some $0 < k < n/2$, and two vertices span an edge if they are disjoint. Our proof works as long as the ratio $k/n$ is bounded away from zero. 

So, as in Swiss cheese, although these graphs have plenty of very large holes (independent sets), there are no ``uniformly sparse sets'': any sparse set is nothing else than a part of one of the big independent sets, with a small perturbation.  

\paragraph{Juntas.}
The context in which we encountered this problem was in trying to characterize independent sets in
graph products. In $K_3^n$, for instance, the largest independent sets are those obtained by fixing a single coordinate (``dictatorships''). Other examples of independent sets that depend on few coordinates (so called ``juntas'') are, say, all vertices that have at least two ``0'' entries in their first three coordinates. 
Similarly, in the Kneser graph, the largest independent sets are determined by a single ``coordinate'', namely, all sets containing a specific element (this is part of the  Erd\H{o}s-Ko-Rado theorem). And, as in $K_3^n$, there exist junta-like independent sets capturing a constant proportion of the vertices, which are determined by few elements.

In \cite{DinurFR06} and \cite{DinurF09} it is proven that {\em any} independent set in these graphs can essentially be captured by a set depending on few coordinates. E.g., for every $\eps>0$ there exist a $\delta > 0$ and a positive integer $j$ such that for any independent (or very sparse) set  $U \subset V(K_3^n)$ there exists a set of coordinates $J \subset [n]$ with $|J| \leq j$ and a set $T \subset \{0,1,2\}^J$ such that all but $\eps 3^n$ of the vertices in $U$ have their $J$-coordinates in $T$. Furthermore, this set $T$ ``explains'' why $U$ is independent, because $T$ itself is extremely sparse in the graph $K_3^J$. A similar statement was proven there for a large class of product graphs, and for the Kneser graphs.

However, there was a fly in this ointment (or a thorn in the sheep-tail-fat, as we say): we conjectured that there must exist a set $T$ as above that is not only sparse in the product graph of dimension $j$, but actually independent, thus providing a  {\em complete} explanation for the sparseness of $U$; i.e., we conjectured that any independent (or very sparse) set is almost completely contained in 
a set depending on few coordinates that is truly independent (as opposed to merely sparse).
In this paper we manage to settle this issue and prove the conjecture.
The key to this is applying the main theorem in this paper, Theorem \ref{thm:main}, to the set $T$, which belongs to the $j$'th power of the base graph. We show that the edge removal phenomenon holds in the product graph, thus $T$ can be slightly altered to produce a {\em truly independent set} (as opposed to a sparse one). 

\paragraph{Related work.}
The structure of our proof is very closely modeled on Fox's proof of the graph removal lemma, \cite{Fox11}, where he improved the longstanding bound on the dependence of the constants in the lemma (most famously, in the triangle removal lemma). One of the differences is that we have, to use Fox's terminology, a shattering lemma that is special to this setting, and is nothing else than a (minor) generalization of the main result of \cite{DinurFR06} (Theorem \ref{thm:bipartitedfr}). It states that whenever two large sets of vertices span few edges between them, it is because there is a small set of coordinates that these two sets are strongly correlated with. This is a consequence of a central theorem in  
\cite{DMR}, which in turn, relies heavily on the invariance principle of \cite{MOO05}.

Our main result is also reminiscent of, and related to, removal lemmas in groups, see \cite{HatamiST13}, \cite{KralSV13},
and for systems of equations over finite fields, \cite{KralSV12}.
 Theorem \ref{thm:Kneser}, which deals with Kneser graphs, is closely related to the work \cite{DasT16}, 
where a removal lemma is proven for the special case of sets that are close in size to a maximal independent set in the Kneser graph $K(n,k)$, for any value of $k$.

\paragraph{Structure of the paper.} 
In Section~\ref{sec:prelim} we present some preliminary definitions.
In Section~\ref{sec:mainthm} we state our main theorem and sketch the proof. 
In Section~\ref{sec:matching} we reduce to a case of a ``matching like'' function, which will eventually lead to a substantial improvement in the resulting bounds.
In Section~\ref{sec:potential} we show that any non-negative function has a good approximation (in a specific sense) by a function that depends on few coordinates. 
In Section~\ref{sec:sparse} we complete the proof of our main theorem.
In Section~\ref{sec:Kneser} we state and prove the Swiss Cheese Theorem regarding Kneser graphs, showing that they too exhibit an edge removal phenomenon. 
Finally, in the appendix we explain how to extend the main theorem from \cite{DinurFR06} to the form that we use in this paper.

%%%%%%%%%%%%%%%%%%%%%%%%%%%%%%%%%%%%%%%%%%%%%%%%%%%%%%%%%%%%%%%%%%%%%%
\section{Preliminaries}\label{sec:prelim}

In the rest of the paper we fix a set $V$ with a reversible, irreducible, aperiodic Markov chain on it given by a matrix
$A$. 
All functions and constants we encounter from now on may depend on $V$ and $A$. 
Let $\mu$ denote the unique stationary measure of $A$ on $V$. We will be working with $V^n, A^{\otimes n}$ and $\mu^{\otimes n}$, often just writing $A$ and $\mu$ as shorthand to avoid cumbersome notation.
Whenever we take expectation of a function on $V^n$ it is according to $\mu$, and we use $\mu$ also to define the standard inner product between functions on $V^n$. We think of the ground set $V^n$ as the vertices of a (product) graph, and of the non-zero-probability transitions as edges. The weight of a (directed) edge $(x,y)$ is 
\[
w_{(x,y)} := \ip{\indic_{x},A\indic_{y}} \; .
\] 
This is the asymptotic probability of a step in the random walk governed by $A$ to traverse $(x,y)$.
Equivalently it is the probability of $(u,w)=(x,y)$, where $u$ is chosen by the stationary distribution, and $w$ is chosen from $u$'s neighbors according to the probabilities dictated by the transition matrix.
We will talk about ``the weight of the edges spanned'' by a set $U$, or a function $f$, simply meaning 
  $\ip{\indic_U,A\indic_U} $, or $\ip{f,Af} $. Consequently, a set $U \subset V^n$ is called {\em independent} if $\ip{\indic_U,A\indic_U} =0$.
We will say that a function $g: V^n \to [0,1]$ is $\eps$-far from independent if for every independent set $U$ 
we have $\E[ \indic_{\overline{U}}\cdot g ] > \eps$.

\paragraph{Capturing sparse sets using juntas.}
A crucial ingredient in our proof is the following variant on the main result from~\cite{DinurFR06}. 
In the appendix we explain how it follows from previous work. 

\begin{theorem}\label{thm:bipartitedfr}
For all $0 < \eps \le 1/2$ there exist $\delta>0, j\ge 1$, such that the following holds.
For all $n \ge 1$ and $f_1,f_2:V^n \to [0,1]$ such that $\ip{f_1,Af_2} \le \delta$, there exist $J \subseteq [n]$, $|J| \le j$, and $T_1,T_2 \subseteq V^J$ such that 
\begin{equation}
\label{eq:outsidetlowexpect}
\E_{x \in V^J}[\indic_{\overline{T_1}}(x) \E [f_1(x,\cdot)]] \le \eps, 
 \qquad \E_{x \in V^J}[\indic_{\overline{T_2}}(x) \E [f_2(x,\cdot)]] \le \eps,
\end{equation}
and $\ip{\indic_{T_1}, A\indic_{T_2}} \le \eps$.
Moreover, one can take $\delta=\delta_1(\eps):= \eps^c$ and $j=j_1(\eps):=\eps^{-c}$
where $c>0$ is a constant depending only on $A$.
\end{theorem}

%%%%%%%%%%%%%%%%%%%%%%%%%%%%%%%%%%%%%%%%%%%%%%%%%
\section{Main Theorem}\label{sec:mainthm}

\begin{theorem}\label{thm:main}
For all $\eps>0$ there exists $\delta >0$ such that the following holds. 
If $g : V^n \to [0,1]$ is $\eps$-far from independent then 
 $\ip{g,Ag} > \delta$. Moreover, we can take $\delta = \delta_2(\eps) := 1/\tower(O( \log(1/\eps)))$.
\end{theorem}

This easily implies the desired strengthening of \cite{DinurFR06} which was the main motivation for this work. 

\begin{cor}\label{cor:ashkara}
For all $\eps>0$ there exist a $\delta>0$ and a positive integer $j$ such that the following holds.
Let $g : V^n \to [0,1]$, with $\ip{g,Ag} < \delta$. Then there exists $J \subset [n]$, with $|J| \le j$, and 
$T \subseteq V^J$ such that 
\begin{enumerate}
\item $\ip{\indic_T,A\indic_T} =0$, and
\item $
\E_{x \in V^J}[\indic_{\overline{T}}(x) \E[g(x,\cdot)]] \le \eps
$.
\end{enumerate}
Moreover, we can take $\delta = \delta_1(\delta_2(\eps/2))$ and $j = j_1(\delta_2(\eps/2))$.
\end{cor}
\begin{proof} 
Invoke Theorem \ref{thm:bipartitedfr} with $f_1=f_2=g$ and $\eps$ taken to be $\delta_2(\eps/2)$ to produce $J$, with $|J| \le j$, 
and $T' \subset V^J$ satisfying that
\[
\E_{x \in V^J}[\indic_{\overline{T'}}(x) \E [g(x,\cdot)]] \le \delta_2(\eps/2) \le \eps/2,
\]
and $\ip{\indic_{T'}, A\indic_{T'}} \le \delta_2(\eps/2)$. Then invoke Theorem~\ref{thm:main} on $\indic_{T'}$ to find an independent set $T \subset T'$ 
with 
\[
\E[\indic_{T'} -\indic_{T}] \leq \eps/2 \; .
\] 
It follows that $T$ satisfies conditions (1) and (2), as required. 
\end{proof}

\subsection{Sketch of proof of the main theorem} 

We begin the sketch for the case of sparse sets (i.e., for $g$ having range $\{0,1\}$), as that captures all the main ideas. We will briefly mention the extension to functions at the end of this subsection.
Given a set $U \subset V^n$ which is $\eps$-far from being independent we wish to show that it spans many edges, i.e., that $\ip{\indic_U,A\indic_U}$ is large.

Notice that any set $I \subset [n]$ naturally defines a partition of $V^n$ into $|V|^{|I|}$ parts according to the coordinates in $I$. We can then study how $\indic_U$ behaves on these parts. If it is constant (zero or one) on 
each part, then $I$ perfectly captures $U$. Letting $W$ be a random variable whose value is the conditional expectation of $\indic_U$ on a random part of the partition, and letting $H:=H(U,I)$ be the expectation of $W \log(W)$, then $H$ is a good indication of how well $I$ captures $U$. When $I =\emptyset$ then $H = \alpha \log(\alpha)$, where $\alpha$ is the measure of $U$. Moreover, $H=0$ if and only if $\indic_U$ is completely determined by the coordinates in $I$. Furthermore, $H$ is always non-positive, and is monotone increasing with respect to refining the partition induced by $I$ by adding further coordinates. 
For a positive integer $r$, an $r$-refinement of a partition induced by $I$ as above, is one which involves adding $r$ new coordinates per every part of the partition, thus adding up to $r \cdot |V|^{|I|}$ new coordinates.

Given a set $U$, beginning with the trivial partition of $V^n$ (corresponding to $I_0 = \emptyset$) we will iteratively apply $r$-refinements, producing $I_0 \subset I_1\subset \cdots$ attempting to substantially increase $H(U,I_i)$ in each step, and stop when this is no longer possible. On the one hand, since $H\le 0$, the number of steps, and hence the total number of coordinates involved in the final partition is bounded from above by some constant $k$. On the other hand, we will show that if $U$ is $\eps$-far from an independent set, and $\ip{\indic_U,A\indic_U} = \delta  < \delta(\eps,k)$ then for any partition coming from at most $k$ coordinates there is an $r$-partition that further increases $H$ substantially. So if $\delta$ is too small this yields a contradiction.
In the proof we will sketch the exact dependence between all parameters  involved (including $r$ and $k$).

The crux of the proof, then, is how one can utilize the fact that $U$ is sparse (i.e., that $\delta$ is small), $U$ is $\eps$-far from being independent, and $|I|$ is not too large, in order to show the existence of an $r$-refinement which substantially increases $H(U,I)$. Our engine for this is (a slight variation on) the result from \cite{DinurFR06}, where the fuel of this engine is the invariance principle of \cite{MOO05} (as applied in \cite{DMR}). Whenever two parts of the partition, say $X$ and $Y$, span few edges between them, our engine will produce a refinement of the partition, according to a bounded number of new coordinates, such that on at least one of the two parts, say $X$, the resulting increase in $H$ will be proportional to the measure of $U \cap X$.  This approach is sufficient to prove our main theorem, with a $\tower(O(1/\eps))$-type dependence between $\eps$ and $\delta$.
 
 However, using a key idea from Fox's improvement to the bounds in the graph removal lemma, \cite{Fox11}, we can cut this dependence down to $\tower(O(\log(1/\eps)))$. This involves replacing the set $U$ with a subset $U'$, the support of a maximal matching contained in it. The advantage of $U'$ is that no independent set contained in it captures more than half of its mass. As the proof will show, this helps avoid slowly whittling away at $U$, and speeds up the relative increase in $H$ in every step.

 One last element in the proof of our main theorem: the use of functions with range $[0,1]$ instead of sets. In the final part of the paper, when applying our result to sets in Kneser graphs, we will ``extrapolate'' these sets  to $[0,1]$-valued functions on $\{0,1\}^n$. Therefore it makes sense for us to make a few small adaptations in our presentation, replacing sets (which may be thought of as functions with range $\{0,1\}$) with functions with range $[0,1]$.  It turns out that this natural variant is not much harder to treat than the original one.
 
%%%%%%%%%%%%%%%%%%%%%%%%%%%%%%%%%%%%%%%%%

\section{Reducing to matching-like functions}\label{sec:matching} 

Here we observe that it suffices to prove Theorem \ref{thm:main} in the special case that $f$ is a \emph{matching-like function}, which is a function not having too much weight on any independent set. The use of this innocuous condition, which replaces the condition of being far from independent, leads to a substantial improvement.

\begin{definition}\label{def:balanced}
Let $f:V^n \to [0,1]$. 
We say that $f$ is \emph{matching-like} if for any
independent set $W \subseteq V^n$, 
\[
\E_x[\indic_W(x) f(x)] \le \E[f]/2.
\]
\end{definition}
The terminology comes from the observation that when $\mu$ is the uniform measure, the indicator function of the vertices touched by any matching is ``matching-like,'' since no independent set can contain both ends of an edge.

The following claim shows that for any function $g: V^n \to [0,1]$  we can find a matching-like function $f$ such that
$g \ge f$ pointwise, and such that the set of vertices $x$ such that $f(x) < g(x)$ is an independent set. 
We note that when the measure $\mu$ is the uniform measure, and $g$ is the indicator of a set $U$, then 
we can simply take any maximal matching inside $U$ and let $f$ be the indicator of its vertices. The three properties below
are then easy to verify.

\begin{claim}\label{clm:maxmatching}
For any $g: V^n \to [0,1]$ there exists a function $f:V^n \to [0,1]$ satisfying that
\begin{enumerate}
\item
$f(x) \le g(x)$ for all $x.$ 
\item
$f$ is matching-like, and
\item
$\{x: f(x) < g(x) \}$ is an independent set. 
\end{enumerate}
\end{claim}

\begin{proof}
Note that the set of all functions $f$ that fulfill conditions (1) and (2) is non-empty (as it contains the identically $0$ function), and compact. Hence there  exists some $f$ in this set which maximizes $\sum f(x)$. This $f$ must also fulfill condition (3), as otherwise we would have an edge $\{a,b\}$ with $f(a) < g(a)$ and $f(b) < g(b)$. If such an edge were to exist, there would exist some small positive constant $\gamma$ such that we could add $\gamma/\mu(a)$ to $f(a)$, and 
$\gamma/\mu(b)$ to $f(b)$, yielding a new function $f'$ that still fulfills condition (1) and (3). This function will  also fulfill condition (2) since its expectation is greater than that of $f$ by $2\gamma$, but its weight on any independent set is 
greater by at most $\gamma$ (because no independent set contains both $a$ and $b$.) So the existence of $f'$ contradicts the maximality of $f$.
\end{proof}

We now state a theorem quite similar to our main one, Theorem \ref{thm:main}, for the special case of matching-like functions, and then, using Claim~\ref{clm:maxmatching}, we can easily deduce Theorem \ref{thm:main}.

\begin{theorem}\label{thm:mainbalanced}
For all $\eps>0$ there exists a $\delta>0$ such that any matching-like $f:V^n \to [0,1]$ with $\E[f] \ge \eps$
satisfies $\ip{f,Af} > \delta$. Moreover, we can take $\delta = 1/\tower(O( \log(1/\eps)))$.
\end{theorem}

This now easily implies our main theorem.

\begin{proof}[Proof of Theorem \ref{thm:main}]
Given $g$ we invoke Claim \ref{clm:maxmatching} to produce an appropriate matching-like $f$.
Note that since the set $U:= \{x: f(x) < g(x)\}$ is independent we have that 
$$
\E[f] \ge \E[ \indic_{\overline{U}}\cdot f ]   = \E[ \indic_{\overline{U}}\cdot g ] > \eps.
$$
So, by Theorem \ref{thm:mainbalanced} we have $\ip{f,Af} > \delta$. Since $g \ge f$ pointwise it follows that 
$\ip{g,Ag} > \delta$ as required.
\end{proof}

\section{The potential argument}\label{sec:potential}

Given a function $f: V^n \to [0,1]$, and a set of coordinates $I \subseteq [n]$ we wish to study how well $f(x)$ is predicted by the coordinates of $x$ indexed by $I$. If $f$ depends only on the coordinates in $I$, then we think of $f$ as being perfectly correlated with $I$, in which case $f$ is constant on every part of the partition of $V^n$ induced by $I$. Otherwise, we can improve this correlation by refining the partition according to additional coordinates. In Definition~\ref{def:shattering} we set a (partially arbitrary) benchmark for how much this refinement improves the correlation, and define any refinement that succeeds as ``substantially improving the correlation''. If for every $x$ in $V^I$ we partition $\{x\} \times V^ {[n] \setminus I}$ according to $r$ additional coordinates we call this an ``$r$-refinement''. Our main goal in this section is the proof of 
Lemma~\ref{lem:shatteringsuffices} that states, roughly, that for any $f$ and $r$ there exists a (not-too-large) set of coordinates $I$ whose correlation with $f$ cannot be substantially improved by $r$-refinement.
Definition~\ref{def:shattering} and the statement of Lemma~\ref{lem:shatteringsuffices} are the only parts of this section used in the rest of the paper. 

\begin{definition}\label{def:shattering}
Consider a function $f:V^n \to [0,1]$ and let $\alpha$ denote $\E[f]$. For $r \ge 1$ and for a set of coordinates  $I \subseteq [n]$, we say that the correlation of $I$ with $f$ can be \emph{substantially improved by $r$-refinement} if there exists a subset $S \subseteq V^I$ for which the following holds:
\begin{enumerate}
\item
$\E_{x \in V^I}[\indic_S(x) \E [f(x,\cdot)]] \ge \alpha/2$, and
\item
For each $x \in S$ there exists $J_x \subseteq [n]\setminus I$ of cardinality at most $r$ and $T_x \subseteq V^{J_x}$ satisfying 
\begin{enumerate}
\item $\Pr[y \in T_x] \le 3/4$, and
\item $\E_{y \notin T_x} [\E [f(x,y,\cdot)]] \le \alpha/8$.
\end{enumerate}
\end{enumerate}
\end{definition}
Just to ensure the notation in ($b$) is clear:  $\E_{y \notin T_x}[h(y)] := \frac{\sum_{y \not \in T_x} \mu_p(y)h(y)}{\Pr[y \not \in T_x]}$.

As mentioned, a set of coordinates $I \subset [n]$ is perfectly correlated with $f$ if for every $x \in V^I$ it holds that $f(x,\cdot)$ is constant. Substantially improving the correlation of $I$ with $f$ by refinement, means that for a portion of inputs $x \in  V^I$, which are responsible for at least half of the expectation of $f$, it holds that by an appropriate choice of $J_x$, (at most $r$ additional coordinates from $[n] \setminus I$), partitioning $\{x\} \times V^{[n] \setminus I}$ according to these additional coordinates yields a partition where in many parts the conditional expectation of $f$ drops substantially. In a sense this implies that $f$ is closer to being constant on the parts of the refined partition of $V^{[n]}$. 

\begin{lemma}\label{lem:shatteringsuffices}
Consider a function $f:V^n \to [0,1]$ and let $\alpha$ denote $\E[f]$. Let $r$ be a positive integer,
and define the function $\Gamma(\ell):=\ell+r|V|^\ell$.
Then there exists a set $J \subseteq [n]$ of cardinality at most $k=k(\alpha,r):=\Gamma^{\circ 128\log(1/\alpha)}(0)$ whose correlation with $f$ cannot be substantially improved by $r$-refinement, where $\Gamma^{\circ t}$ denotes the composition of $\Gamma$ with itself $t$ times.
\end{lemma}

The proof uses a potential argument. Define the function $\varphi:\R^{\ge 0} \to \R$ by $\varphi(x)=x\log x$ (and set 
$\varphi(0)=0$.)
This is a convex function and is non-positive on $[0,1]$. For $f:V^n \to [0,1]$ and $I \subseteq [n]$ define the \emph{entropy}
of $f$ with respect to $I$ as
\[
H(f,I) := \E_{x \in V^I} [\varphi(\E [f(x,\cdot)])].
\]
Since $\varphi$ is convex, we get from Jensen's inequality that $H(f,I)$ is monotone in $I$, i.e., 
$H(f,I) \le H(f,J)$ whenever $I \subseteq J$.
Moreover, $H(f,I) \le 0$ for all $I$.

\begin{claim}\label{clm:potentialimproves}
Let $f:V^n \to [0,1]$ be a function and $\alpha=\E[f]$. Let $I \subseteq [n]$. If the correlation of $I$ with $f$ is substantially improved by $r$-refinement for some $r \geq 1$ then
\[
H(f,J) \ge H(f,I) +\alpha/128,
\]
where $J := I \cup \bigcup_{x \in S} J_x$, and $S,J_x$ are as in Definition~\ref{def:shattering}.
\end{claim}

We first observe that this claim implies Lemma~\ref{lem:shatteringsuffices}.

\begin{proof}[Proof of Lemma~\ref{lem:shatteringsuffices}]
Notice that by definition $H(f,\emptyset)=\alpha \log \alpha$. Hence, starting with $I_0 =\emptyset$,
we repeatedly apply Claim~\ref{clm:potentialimproves}, and create an $r$-refinement, each time obtaining a new set $I_{t+1}$ with
 $H(f,I_{t+1}) \ge H(f,I_t) + \alpha/128$, and $ |I_{t+1}| \leq \Gamma(|I_t|)$. We continue as long as the correlation of 
 $I_{t+1}$ with $f$ can be substantially improved by $r$-refinement.
Since $H(f,I) \le 0$ for all $I$, this process must terminate after at most $128 \log(1/\alpha)$ steps,  implying the lemma. 
 \end{proof}

\begin{proof}[Proof of Claim~\ref{clm:potentialimproves}] 
Let $S,J_x,T_x$ be as in Definition~\ref{def:shattering}. 
Define $S' \subseteq S$ as the set of $x \in S$ satisfying $\E[f(x,\cdot)] \ge \alpha/4$.
By Item 1 of Definition~\ref{def:shattering},
\begin{equation}\label{eq:potentialimproves1}
\E_{x \in V^I}[\indic_{S'}(x) \E[f(x,\cdot)]] \ge \E_{x \in V^I}[\indic_{S}(x) \E[f(x,\cdot)]] -\alpha/4 \ge\alpha/4.
\end{equation}
Now,
\begin{align}
H(f,J) &= \E_{x \in V^I} \Big[H\Big(f(x,\cdot), \bigcup_{s \in S} J_s\Big)\Big] \nonumber \\
&=  \E_{x \in V^I} \Big[\indic_{S'}(x) H\Big(f(x,\cdot), \bigcup_{s \in S} J_s\Big)\Big] +
       \E_{x \in V^I} \Big[\indic_{\overline{S'}}(x) H\Big(f(x,\cdot), \bigcup_{s \in S} J_s\Big)\Big] .
 \label{eq:twoterms}
\end{align}
By monotonicity, the second term in~\eqref{eq:twoterms} is at least
\begin{align}\label{eq:secondterm}
\E_{x \in V^I} [\indic_{\overline{S'}}(x) \varphi(\E[f(x,\cdot)])].
\end{align}
For analyzing the first term in~\eqref{eq:twoterms}, fix any $x \in S'$. Then using monotonicity,
\begin{align}
H\Big(f(x,\cdot), \bigcup_{s \in S} J_s\Big) &\ge
 H(f(x,\cdot), J_x) \nonumber \\
&= \E_{y \in V^{J_x}} [\varphi(\E[f(x,y,\cdot)])] \nonumber \\
&= \Pr[y \notin T_x] \E_{y \notin T_x} [\varphi(\E[f(x,y,\cdot)]) ] + 
   \Pr[y \in T_x] \E_{y  \in T_x} [\varphi(\E[f(x,y,\cdot)]) ] \nonumber \\
&\ge \Pr[y \notin T_x] \varphi( \E_{y \notin T_x} [\E[f(x,y,\cdot)] ]) +
 \Pr[y \in T_x] \varphi( \E_{y  \in T_x} [\E[f(x,y,\cdot)] ]) \nonumber \\
	&= \lambda \varphi(u) + 
   (1-\lambda) \varphi(v)\nonumber \\
	& \ge  \varphi(w) + w/32,\label{eq:firstterm}
\end{align}
where we define $\lambda := \Pr[y \notin T_x]$, $u := \E_{y \notin T_x}[\E[f(x,y,\cdot)] ]$,
 $v :=\E_{y  \in T_x}[\E[f(x,y,\cdot)] ]$, and $w:=\lambda u + (1-\lambda)v = \E[f(x,\cdot)]$.
The last inequality follows from Claim~\ref{clm:varphi} below,
noting that by Definition~\ref{def:shattering}, $\lambda \ge 1/4$ (Item 2.a), $u \le \alpha/8$ (Item 2.b), and 
$w \ge \alpha/4$ by our choice of $S'$, and therefore $u \le w/2$ as required.
Plugging this into~\eqref{eq:twoterms}, we obtain that
\begin{align*}
H(f,J) &\ge
\E_x[\indic_{S'}(x) \varphi(\E[f(x,\cdot)])] + \E_x[\indic_{\overline{S'}}(x) \varphi(\E[f(x,\cdot)])]
 + \frac{1}{32} \E_x[\indic_{S'}(x)\E[f(x,\cdot)]] \\
&= H(f,I) + \frac{1}{32} \E_x[\indic_{S'}(x)\E[f(x,\cdot)]] \\
&\ge H(f,I) + \frac{\alpha}{128},
\end{align*}
where the last inequality uses~\eqref{eq:potentialimproves1}.
\end{proof}

\begin{claim}
\label{clm:varphi}
For any $u,v>0$ and $\lambda \in [1/4,1]$ satisfying that $u \le w/2$ where $w := \lambda u + (1-\lambda) v$,
\[
\lambda \varphi(u) + (1-\lambda) \varphi(v) \ge \varphi(w) + w / 32 \; .
\] 
\end{claim}
\begin{proof}
From the definition of $\varphi$ we have
\begin{equation}
\label{eq:varphione}
	 \lambda \varphi(u) + 
   (1-\lambda) \varphi(v) =
	 \varphi(w) + w \big(\lambda \varphi(u/w) + (1-\lambda) \varphi(v/w)\big) \; .
\end{equation}
Notice that $v/w = (1- \lambda u/w) / (1-\lambda) \ge 7/6$. Since $\varphi$ is convex,
the line segment connecting
$(u/w,\varphi(u/w))$ with $(v/w,\varphi(v/w))$ 
lies above the line segment connecting
$(1/2,\varphi(1/2))$ with $(7/6,\varphi(7/6))$ (see Figure~\ref{fig:varphi}),
and therefore, their points of intersection with the vertical line $x=1$ satisfy  
\begin{equation}
\label{eq:varphitwo}
\lambda \varphi(u/w) + (1-\lambda) \varphi(v/w) \ge 
\frac{1}{4} \varphi(1/2) + \frac{3}{4} \varphi(7/6) > 1/32 \;.
\end{equation}
Combining Eqs.~\eqref{eq:varphione} and~\eqref{eq:varphitwo} yields the result.
\end{proof}
\begin{figure}[ht]
\begin{center}
\includegraphics[width= 0.4\textwidth]{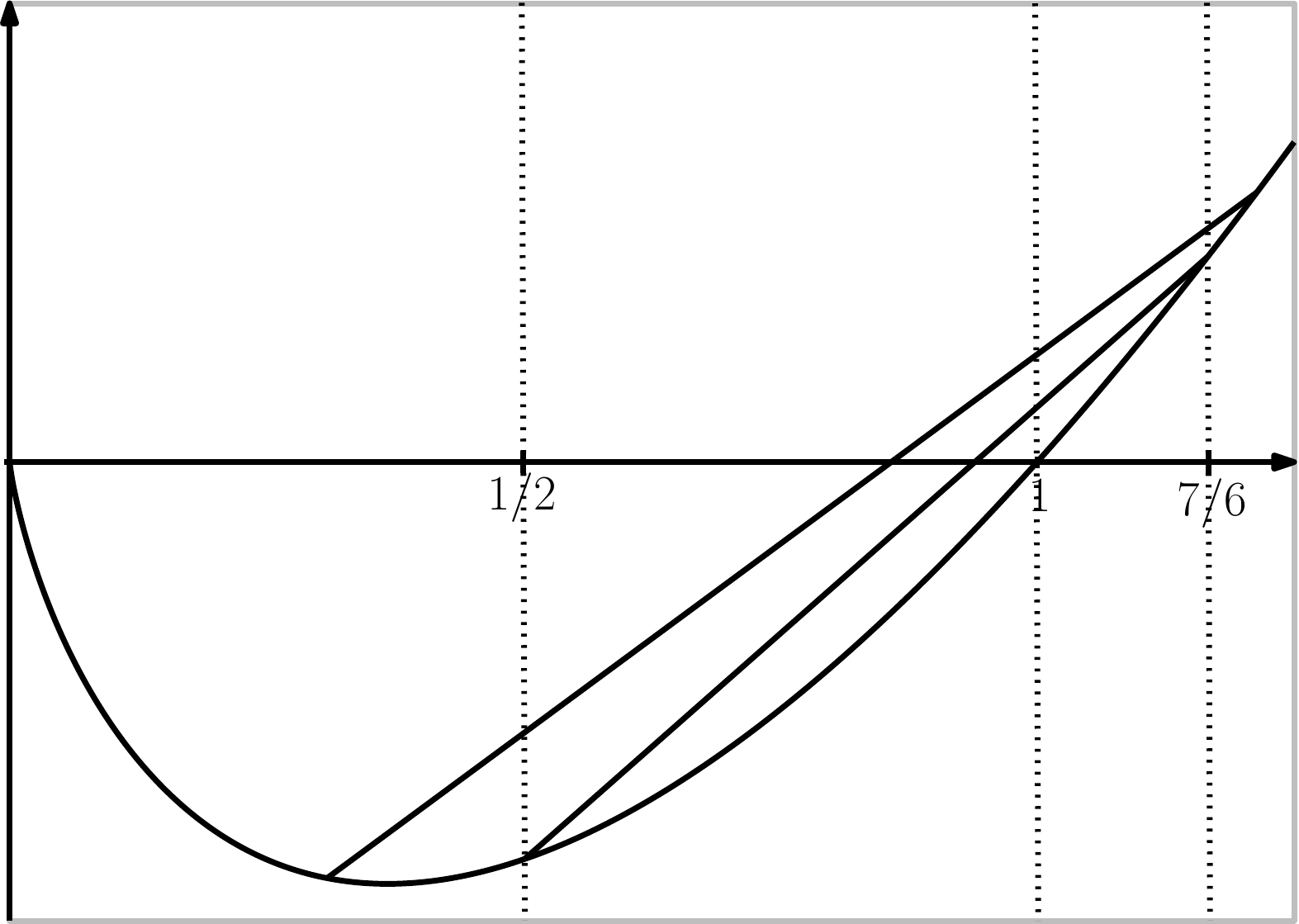} 
\caption{\label{fig:varphi} 
$\varphi(x)$ for $x \in [0,5/4]$.}
\end{center}
\end{figure}

\section{Improving correlation through refinements}\label{sec:sparse}

In this section we prove Theorem~\ref{thm:mainbalanced}, thereby completing the proof of our main theorem.  
It will be convenient to use the following corollary of Theorem~\ref{thm:bipartitedfr}, showing
that if two functions $f_1,f_2$ span very few edges, then one of them must be concentrated on a junta of measure $3/4$.
\begin{cor}\label{cor:shattering}
For all $0 < \eps \le 1/2$ there exist $\delta>0, r \ge 1$, such that the following holds.
For all $n \ge 1$ and $f_1,f_2:V^n \to [0,1]$ such that $\ip{f_1,Af_2} \le \delta$, there exist $i \in \{1,2\}$, $J \subseteq [n]$, $|J| \le r$, and $T \subseteq V^J$ such that 
\begin{equation}\label{eq:corshattering}
\E_{x \in V^J}[\indic_{\overline{T}}(x) \E [f_i(x,\cdot)]] \le \eps,
\end{equation}
and 
\begin{equation}\label{eq:corshattering2}
\Pr[x \in T] \le 3/4.
\end{equation}
Moreover, one can take $\delta=\delta_1(\eps):= \eps^c$ and $r=r_1(\eps):=\eps^{-c}$
where $c>0$ is a constant depending only on $A$.
\end{cor}
\begin{proof}
Apply Theorem~\ref{thm:bipartitedfr}, yielding  $|T_1|, |T_2| \leq r$ where $r=j$, and observe that $\ip{\indic_{T_1}, A\indic_{T_2}} \le \eps \le 1/2$
implies that either $\Pr[x \in T_1] \le 3/4$ or $\Pr[x \in T_2] \le 3/4$.
\end{proof}

Let $w_{\min}$ be the minimal weight of an edge in the transition graph of $A$, i.e., the minimal positive value of $\ip{\indic_u,A \indic_v}$, where $u$ and $v$ are elements of $V$. Note that the weight of the minimal edge for $A^{\otimes m}$ is $w_{\min}^m$.
\begin{lemma}\label{eq:shatterer}
For all $\eps>0$ there exists $r \ge 1$, such that for any  $k \ge 1$ there exists $\delta>0$ such that for any matching-like $f:V^n \to [0,1]$ with
$\ip{f,Af} \le \delta$ and $\E[f] = \eps$, and for any $I \subseteq [n]$ of cardinality at most $k$, the correlation of $I$ with $f$  can be substantially improved by $r$-refinement. Moreover, we can take 
$\delta=\delta_2(\eps,k):= w_{\min}^k \delta_1(\eps/32) = w_{\min}^k (\eps/32)^c$ and 
$r=r_2(\eps):= r_1(\eps/32) =  (\eps/32)^{-c}$.
\end{lemma}
Notice that $r_2$ depends on $\eps$ but not on $k$. Also note that the factor $w_{\min}^k$, which, as we will see in the proof, is incurred by the normalization when moving from $V^n$ to $V^{n-k}$, is the dominant factor in our calculations (as we will need $k = \tower(O(\log(1/\eps)))$).
\begin{proof}
Fix a function $f$ and a set $I \subseteq [n]$, with $|I| \le k$ as in the statement of the lemma.
We construct a set $S \subseteq V^I$ and sets $J_x,T_x$ satisfying the requirements in Definition~\ref{def:shattering} as follows. 
Initially, we set $S=\emptyset$.
We consider all edges (including loops!) $(x_1,x_2)$ (i.e., all pairs $(x_1,x_2) \in V^I \times V^I$ with $w_{(x_1,x_2)} > 0$), in an arbitrary order. 
For each edge $(x_1,x_2)$
if either $x_1$ or $x_2$ is already in $S$, we continue to the next edge. 
Otherwise, we apply Corollary~\ref{cor:shattering} to the functions $f(x_1,\cdot)$ and $f(x_2,\cdot)$ 
with $\eps$ taken to be $\eps/32$, resulting in $i \in \{1,2\}$, $J \subseteq [n] \setminus I$, with $|J| \leq r_1(\eps/32)$, and  $T \subseteq V^J$. 
We then add $x_i$ to $S$,
and define $J_{x_i}$ to be $J$ and $T_{x_i}$ to be $T$. 
This completes the description of the construction. 
Notice that we are allowed to apply Corollary~\ref{cor:shattering} above with parameter $\eps$ set to $\eps/32$, since
\[
\ip{f(x_1,\cdot), Af(x_2,\cdot)} \le 
w_{\min}^{-|I|} \ip{f, Af} \le w_{\min}^{-k} \delta =
\delta_1(\eps/32) \; .
\]
Moreover, notice that $\overline{S}$ forms an independent set (since $S$ must contain at least one vertex of each edge), and since $f$ is matching-like, we have 
\[
\E_{x \in V^I}[\indic_{S}(x) \E[f(x,\cdot)]] = \eps - \E_{x \in V^I}[\indic_{\overline{S}}(x) \E[f(x,\cdot)]] \ge \eps/2.
\]
Finally, for all $x \in S$ we have by~\eqref{eq:corshattering2} that $\Pr[y \in T_x] \le 3/4$, and 
by~\eqref{eq:corshattering} that
\[
\E_{y \notin T_x}[ \E[f(x,y,\cdot)] ] = 
 \E_y[ \indic_{\overline{T_x}}(y)\E[f(x,y,\cdot)]  ] / \Pr_y[y \notin T_x] \le \frac{\eps/32}{1/4}
 = \eps/8.
\]
We conclude that $S$, $\{J_x\}$, and $\{T_x\}$ satisfy the requirements in Definition~\ref{def:shattering}, as required.
\end{proof}

Theorem~\ref{thm:mainbalanced} follows from Lemma~\ref{lem:shatteringsuffices} and Lemma~\ref{eq:shatterer}.
\begin{proof}[Proof of Theorem~\ref{thm:mainbalanced}]
Given a matching-like $f$, with $\E[f]=\eps$, let $r = r_2(\eps)$ be as given by Lemma~\ref{eq:shatterer}. 
Apply Lemma~\ref{lem:shatteringsuffices} with $f$ and $r$ to get a set $J$ of cardinality at most 
$k=k(r,\eps)= \tower(O(\log(1/\eps)))$ 
such that the correlation of $J$ with $f$ cannot be substantially improved by $r$-refinement. From Lemma~\ref{eq:shatterer} it follows that $\ip{f,Af} > \delta$, with $\delta=\delta_3(\eps,k) = (\tower(O(\log(1/\eps))))^{-1}$.
\end{proof}

%%%%%%%%%%%%%%%%%%%%%%%%%%%%%%%%%%%%

\section{Kneser graphs are like Swiss cheese}\label{sec:Kneser}

In this section we prove Theorem~\ref{thm:Kneser}, which extends our main theorem to the case of Kneser graphs. 
Fix $0 < p <1/2$. We will consider the Markov chain on $\{0,1\}^n$  
which moves independently on each coordinate according to the transition matrix 
$$
 \left(\begin{array}{cc}\frac{1-2p}{1-p} & \frac{p}{1-p} \\ 1 & 0\end{array}\right). 
 $$
 The stationary measure of this Markov chain is the product measure $\mu_p = (1-p,p)^{\otimes n}$, and all transitions $(x,y)$ have probability 0 if $x$ and $y$ are not disjoint. If $x$ and $y$ are disjoint, then the weight of the edge $(x,y)$ is precisely $p^{|x|}p^{|y|}(1-2p)^{n-|x|-|y|}$. So, for two disjoint sets $x, y \subset [n]$, we define
 $$
 \mu_{p,p}(x,y):= p^{|x|}p^{|y|}(1-2p)^{n-|x|-|y|}.
 $$
Recall the following notation. Given a set of coordinates $J \subset [n]$, and two vectors $w \in \{0,1\}^J$, and $x \in \{0,1\}^{[n] \setminus J}$, we will write $(w,x)$ for the element of $\{0,1\}^n$ formed by merging them appropriately.
 The following, then,  is a special case of Corollary \ref{cor:ashkara}.
\begin{theorem}\label{thm:cube}
Let $ 0 < p < 1/2$. There exists  functions $\delta_p = \delta: [0,1] \to [0,1]$ and  $j_p = j : [0,1] \to \mathbb{N} $ such that the following holds. Let    $g : \{0,1\}^n \to [0,1]$, and let  
$$
\edge(g) = \sum_{x \cap y = \emptyset} g(x)g(y)\mu_{p,p}(x,y).
$$
Then for every $\eps \in [0,1]$, if $\edge(g) \leq \delta(\eps)$ then there exists $J \subset [n]$ with $|J| \leq j(\eps)$, and 
$T \subset \{0,1\}^J$ such that
\begin{enumerate}
\item $T$ is an  intersecting family.
\item
\[
\E_{w \in \{0,1\}^J}[\indic_{\overline{T}}(w) \E_{x \in \{0,1\}^{[n]\setminus J}}[g(w,x)]] \le \eps, 
\]
where all expectations are taken with respect to $\mu_p$. 
\end{enumerate}
\end{theorem}

The main theorem of this section is very similar, except it is set on a single layer of the cube $\{0,1\}^n$, i.e., on $\binom{[n]}{k}$. 

\begin{theorem}\label{thm:Kneser}
Let $ 0 < p < 1/2$ and let $\delta_p = \delta: [0,1] \to [0,1]$ and  $j_p = j : [0,1] \to \mathbb{N} $ be as in Theorem \ref{thm:cube}.  Let  $n$ and $k=pn$ be positive integers, and let  $G=G(n,k)$ be the Kneser graph with
$$
V(G) = \binom{[n]}{k}, E(G) = \{ \{x,y\} : x \cap y = \emptyset\}.
$$
Let $f : V(G) \to [0,1]$, and let  
\[
\edge(f):= \sum_{\{x,y\} \in E(G) } \frac{f(x)f(y)}{\binom{n}{k}\binom{n-k}{k}}  .
\]

Then for every $\eps \in [0,1]$, and $\edge(f) \leq \delta(\eps)$ then if $n$ is sufficiently large there exists $J \subset [n]$ with $|J| \leq j(\eps)$, and 
$T \subset \{0,1\}^J$  such that
\begin{enumerate}
\item $T$ is an intersecting family, and
\item 
$$
{\binom{n}{k}}^{-1} \sum_{(x \cap J) \not \in T} f(x)  \le  5 \eps \; .
$$
\end{enumerate}
\end{theorem}

We will show how to deduce Theorem \ref{thm:Kneser} from Theorem \ref{thm:cube}.
First, we need two lemmas, regarding moving from functions on a single layer to functions on the whole cube, and vice versa.
\begin{lemma}[The Up Lemma]\label{lem:up}
Let $k = p n$ for $0 < p < 1/2$.
For $f: \binom{[n]}{k}\to [0,1] $, define $g: \{0,1\}^n \to [0,1] $ by
$$
g(x): = \left\{\begin{array}{cc}
\binom{|x|}{k}^{-1} \sum_{x' \subseteq x} f(x')  & |x| \ge k
\\0 & \mbox{otherwise.} \end{array}
\right. 
$$ 
Then 
\[
\edge(g) \le \edge(f) \; .
\]

\end{lemma}
\begin{proof}
For any $f$ and the corresponding $g$,
\begin{align*}
 \edge(g)& = \sum_{x \cap y = \emptyset}\mu_{p,p}(x,y) \left(\binom{|x|}{k}\binom{|y|}{k}\right)^{-1}\sum_{x' \subseteq x, y' \subseteq y} f(x')f(y') \\
&=
 \sum_{ x' \cap y' = \emptyset}c(p,n) \frac{ f(x')f(y')}{\binom{n}{k}\binom{n-k}{k}}=c(p,n)  \edge(f) \;, 
\end{align*}
 where 
\[
c(p,n) =  \sum_{x \cap y = \emptyset, x' \subseteq x, y' \subseteq y}\mu_{p,p}(x,y) \frac{\binom{n}{k}\binom{n-k}{k}} {\binom{|x|}{k}\binom{|y|}{k}}
\]
is independent of the pair $(x',y')$ and of the function $f$.
 Plugging in the case $f\equiv 1$, where $\edge(f)=1$, $g(x)={\bf 1}_{|x| \ge k}$ gives
 $$
 c(p,n) = \edge(g) \le 1,
 $$
since $g \le 1$
\end{proof}

\begin{lemma}[The Down Lemma]\label{lem:down}
Let $k=pn$, let $f$ and $g$ be as above, and let $J \subset [n]$, and $w \in \{0,1\}^J$. 
Define
$$
V_w(g):=  p^{|w|}(1-p)^{|J|-|w|}\E_{x \in \{0,1\}^{[n] \setminus J}} [g(w,x)]= \sum_{x \in \{0,1\}^{[n] \setminus J}}  g(w,x)\mu_p(w,x) \; .
$$
Let
$$
V_w(f) := \sum_x \frac {f(w,x)}{\binom{n}{k}},
$$
where the sum is over $x \in \{0,1\}^{[n] \setminus J}$ of size precisely $k - |w|$.
Then, for sufficiently large $n$,
 $$ V_w(f) \le  5 V_w(g) \; .$$
\end{lemma}
\begin{proof}
Throughout this proof, when summing over $x$, we are restricting ourselves to the case $|x| + |w| \ge k$, since other values of $x$ contribute nothing.
Observe that
\begin{align}
V_w(g) 
&= 
\sum_x g(w,x)\mu_p(w,x) \nonumber \\
&=
\sum_x \frac{\mu_p(w,x)}{\binom{|x|+|w|}{k}} \sum_{w' \subseteq w, x' \subseteq x} f(w',x') \nonumber  \\
&\ge 
\sum_x \frac{\mu_p(w,x)}{\binom{|x|+|w|}{k}} \sum_{x' \subseteq x} f(w,x')  \nonumber  \\
% & \mbox{ (Reversing the order of summation),} \nonumber \\
&= \sum_{x'} \frac{f(w,x')}{\binom{n}{k}} \left( \sum_{x \supseteq x'}  \frac{\mu_p(w,x)\binom{n}{k}}{\binom{|x|+|w|}{k}}  \right) \; .\label{eq:vgvf} 
\end{align}
where in the last equality we reversed the order of summation. 
Since the first sum in Eq.~\eqref{eq:vgvf} is precisely $V_w(f)$, it suffices
to show that the second sum (which only depends on $p$, $|J|$, and $|w|$) 
is at least $1/5$. 
To this end observe that
$$
 \sum_{x \supseteq x'}  \frac{\mu_p(w,x)\binom{n}{k}}{\binom{|x|+|w|}{k}}  = 
 \sum_{i=0}^{n-k-(|J|-|w|)} p^{k+i}(1-p)^{n-k-i}\binom{n}{k+i} \frac{((n-k)-(|J|-|w|))_i}{(n-k)_i}.
$$
First note that  $\sum_{i=0}^{ \log(n) \sqrt{np(1-p)}} p^{k+i}(1-p)^{n-k-i}\binom{n}{k+i}$ tends to $1/2$ by the central limit theorem.
Next, for any $i <  \log(n) \sqrt{np(1-p)} $ we have  $\frac{((n-k)-(|J|-|w|))_i}{(n-k)_i} \sim1$ and in particular, for sufficiently large $n$, and $i$ in that range,
$$
\frac{((n-k)-(|J|-|w|))_i}{(n-k)_i}  > 1/2.
$$
So
 $$
 V_w(g) \ge V_w(f) (1/4 -o(1)) .
 $$
\end{proof}

Using the up-lemma  and the down-lemma  we now deduce Theorem \ref{thm:Kneser} from Theorem \ref{thm:cube}.

\begin{proof}[Proof of Theorem \ref{thm:Kneser}]
Let $p$ and $f$ be as in the statement of the theorem, let $\eps \ge 0$ and assume $n$ is sufficiently large and $\edge(f) \le \delta(\eps)$, where 
$\delta(\eps)$ is as defined in Theorem \ref{thm:cube}.  
Let $g$ be as given by the up lemma, Lemma \ref{lem:up}. 
Then $\edge(g) \le \edge(f) \le \delta(\eps)$.
Now, invoke Theorem \ref{thm:cube} to produce $J \subset [n]$ and an intersecting family $T \subset \{0,1\}^J$ which captures $g$, i.e.,
 $$
\E_{w \in \{0,1\}^J}[\indic_{\overline{T}}(w) \E_{x \in \{0,1\}^{[n]\setminus J}}[g(w,x)]] \le \eps, 
$$
or, in other words
$$
\sum_{w\not \in T} V_w(g) \leq \eps \; .
$$
By the down lemma, Lemma \ref{lem:down}, for every $w \in \{0,1\}^J$ (and specifically for $w \not \in T$) 
we have 
$$V_w(f) \le 5 V_w(g)$$
so
$$
\sum_{w\not \in T} V_w(f)   \le 5\eps
$$
as required.
\end{proof}

\appendix
\section*{Appendix}

\section{Theorem~\ref{thm:bipartitedfr}}
\label{sec:dfr}

Theorem~\ref{thm:bipartitedfr} is basically the main result of~\cite{DinurFR06},
apart from some minor differences, the most significant of which being 
that we improve the quantitative dependence of the 
parameters (namely, the functions $\delta_1$ and $j_1$) using the work of Dinur and Shinkar~\cite{DinurS10}. 
For the reader's convenience, we include a proof sketch in Section~\ref{sec:dfrproof}.

Alternatively, we now explain how to derive Theorem~\ref{thm:bipartitedfr} from the original statement in~\cite{DinurFR06}, 
which now follows. 

\begin{theorem}[{\cite[Theorem 1.1 + 2nd and 4th remarks there]{DinurFR06}}]\label{thm:originaldfr}
For all $\eps>0$ there exist $\delta>0, j \ge 1$, such that the following holds.
For all $n \ge 1$ and $f:V^n \to \{0,1\}$ such that $\ip{f,Af} \le \delta$, there exist $J \subseteq [n]$, $|J| \le j$, and $T \subseteq V^J$ such that 
\[
\E_{x \in V^J}[\indic_{\overline{T}}(x) \E [f(x,\cdot)]] \le \eps, 
\]
and $\ip{\indic_{T}, A\indic_{T}} \le \eps$.
\end{theorem}

The differences between this and our Theorem~\ref{thm:bipartitedfr} are as follows.
First, our theorem considers functions with range $[0,1]$ as opposed to $\{0,1\}$. 
The proof in~\cite{DinurFR06} actually applies to the more general case, as is easy to check. 
Alternatively, one can derive the more general case from the restricted one by replacing a function 
$f:V^n \to [0,1]$ with the function $f':V^{n+m}\to \{0,1\}$ where we define
$f'(x,y)$ to be $1$ with probability $f(x)$ and $0$ otherwise, independently over all $x,y$.
Then as $m$ goes to infinity, $\ip{f',Af'}$ converges to $\ip{f,Af}$ and similarly for the other
expressions appearing in the theorem. 

A second difference is that our theorem involves two functions $f_1,f_2$ as opposed to just one as above.  
The proof in~\cite{DinurFR06} can easily be modified to handle this. Alternatively, as before, 
we can derive this from the original statement as follows. Let $a_1,a_2$ be two elements of $V^2$ 
that are connected by an edge and have no self loops. (Such two elements must exist unless
we are in the case in which there is a loop on all vertices in $V$, which means there are no non-empty independent sets in any power of $V$, so this case is irrelevant for our current discussion.)
Then given $f_1,f_2:V^n \to [0,1]$ we define $f:V^{n+2} \to [0,1]$
by taking $f(a,x)$ to be $f_1(x)$ if $a=a_1$, $f_2(x)$ if $a=a_2$, and $0$ otherwise. 
Then $\ip{f,Af}=2w \ip{f_1,Af_2}$ where $w$ is the weight of the edge connecting $a_1$ to $a_2$. 

The final and most significant difference is that Theorem~\ref{thm:originaldfr} does not 
explicitly specify the dependence of $\delta$ and $j$ on $\eps$. 
Inspecting the proof in~\cite{DinurFR06} reveals that the dependence is superpolynomial.
By using an improvement by Dinur and Shinkar~\cite{DinurS10} of the technical statement 
from~\cite{DMR}, we are able to obtain a polynomial dependence of the parameters,
as stated in Theorem~\ref{thm:bipartitedfr}.
We remark that this improvement
has no noticeable effect on the final bound in our main result and we could have used the original bound implicit in~\cite{DinurFR06}; 
we decided to include the improvement as it might be useful for future work. 

\newcommand{\gd}[0]{\delta }
\def\moo{{\textsc{moo}}}
\newcommand{\Emu}[0]{{\mathbb{E}}_\mu}
\newcommand\ipmu[1]{{\langle {#1} \rangle}_{\mu}}
\def\inf{\mathrm{Inf}}

In slightly more detail, the parameters in the proof of Theorem 1.1 in~\cite{DinurFR06} 
all depend polynomially on the functions $\tau_\moo$ and $\delta_\moo$ defined in Theorem~2.2 there.
Those functions can be taken to be polynomial, as shown in the following lemma. 

\begin{lemma}[{Quantitative version of \cite[Theorem~2.2]{DinurFR06}}]
\label{lem:quantmoo}
There exist functions
$\delta_{\moo}(\eps) > 0$ and $\tau_{\moo}(\eps) > 0$
such that for any $\eps >0$, $n\ge 1$, and functions $g_1,g_2 : V^n \to [0,1]$
with $\E[g_1] \ge \eps$, $\E[g_2] \ge \eps$ and
$\ip{g_1, Ag_2} < \delta_{\moo}(\eps) $, there exists a coordinate
$i$ with influence greater than $\tau_{\moo}(\eps)$ on both functions, i.e.,
\[
\inf_i(g_1) > \tau_{\moo}(\eps) \mbox{~~and~~} \inf_i(g_2) > \tau_{\moo}(\eps).
\]
Moreover, one can take 
$\delta_{\moo}(\eps) = \eps^c$ and $\tau_{\moo}(\eps) = \eps^c$ for
some constant $c>0$ depending only on $A$.
\end{lemma}

This lemma is an immediate corollary of~\cite[Theorem 3.1]{DinurS10}, and is derived in precisely the same
way that~\cite[Theorem~2.2]{DinurFR06} is derived from~\cite[Theorem 3.1]{DMR}.
(In fact, a statement nearly identical to Lemma~\ref{lem:quantmoo} already appears as
Corollary 3.2 in~\cite{DinurS10}.)
In more detail and using the notation from~\cite{DinurS10}, to derive the lemma above, 
apply~\cite[Theorem 3.1]{DinurS10} in the contrapositive with some $\rho'>\rho$, say $\rho' = \rho^{1/2}$,
notice that $\langle F_\eps , U_{\rho'}(1-F_{1-\eps}) \rangle_\gamma$ is at least $\eps^C$
by~\cite[Eq.~(2)]{DinurS10}, and that by definition $\inf_i^{\le k}(f) \le \inf_i(f)$.

\subsection{Proof of Theorem~\ref{thm:bipartitedfr}}
\label{sec:dfrproof}

Here we include a proof sketch of Theorem~\ref{thm:bipartitedfr},
closely following the original proof in~\cite{DinurFR06} and occasionally borrowing from the notation there. 

\begin{claim}[{\cite[Lemma 2.3.6]{DinurFR06}}]
\label{clm:influencenoisy}
For any $\eta \in [0,1]$ and function $f:V^n \to [-1,1]$, 
the sum of influences of the ``noisy function'' $N_\eta f$ satisfies
\[
\sum_{i=1}^n \inf_i(N_\eta f) \le (1-\eta^2)^{-2} \; .
\]
In particular, the number of variables that have influence at least $\tau$ on $N_\eta f$ is at most $(1-\eta^2)^{-2} / \tau$. 
\end{claim}

\begin{lemma}[Two-function variant of {\cite[Lemma 2.5]{DinurFR06}}]
\label{lem:perturbnoiseip}
Let $\lambda = \lambda(A) < 1$ be the second absolute eigenvalue of $A$, and let $1-\lambda < \eta < 1$
be sufficiently close to $1$ so that
\begin{equation}
\label{eq:noisyperturb}
(1-\eta) \log_\lambda(1-\eta) \le \sqrt{1-\eta} \;.
\end{equation}
Then for any $f_1,f_2:V^n \to [-1,1]$, 
\[
\Big| \ip{f_1, Af_2} - \ip{g_1, Ag_2} \Big| \le \sqrt{1-\eta} \; ,
\]
where $g_i=N_\eta f_i$.
\end{lemma}
\begin{proof}
By decomposing the functions according to the eigenbasis of $A$,
\begin{align*}
\Big| \ip{f_1,A f_2} - \ip{g_1,A g_2} \Big| 
&= 
\Big| \sum_S (\hat{f}_1(S)\hat{f}_2(S) - \hat{g}_1(S)\hat{g}_2(S)) \lambda_S \Big| \\
&= \Big| \sum_S \hat{f}_1(S)\hat{f}_2(S) (1-\eta^{2|S|}) \lambda_S \Big| \\
&\le 
\sum_S | \hat{f}_1(S)\hat{f}_2(S)| \cdot \max_S (1-\eta^{2|S|}) |\lambda_S| \; .
\end{align*}
A straightforward calculation (see {\cite[Lemma 2.5]{DinurFR06}}) shows that the 
above maximum is at most $\sqrt{1-\eta}$. We can therefore complete the proof
by noting using Cauchy-Schwarz that $\sum_S | \hat{f}_1(S)\hat{f}_2(S)| \le \|f_1\|_2 \|f_2\|_2 \le 1$.
\end{proof}

\begin{lemma}[{\cite[Lemma 2.8]{DinurFR06}}]\label{lem:hypercontra}
There exists a $p=p(A)>2$ such that the following holds. 
For any $\eps > 0$, $j,\ell \ge 1$, and any function $L$ mapping each vertex $a \in V^j$
to a subset $L(a) \subseteq \N$ with $|L(a)| \le \ell$ and satisfying that for at least an $\eps$
measure of pairs $(a,b)$ in $V^j$, $L(a) \cap L(b) \neq \emptyset$, there exists an $i \in \N$ such that
\[
\mu(\{ a \in V^j \, : \, i \in L(a) \}) \ge (\eps/\ell^2)^{2p/(p-2)} \; .
\]
\end{lemma}

\begin{lemma}[{\cite[Claim 3.1]{DinurFR06}}]
\label{lem:dfrclmonnoiseoutside}
For any $\eps > 0$, $\eta \le 1$, $j \ge 1$, and $f : V^j \to [0,1]$,
\[
\E_{x \in V^j}[\indic_{N_\eta f(x) \le \eps} f(x)] \le \eps \; .
\]
\end{lemma}

\begin{proof}[Proof of Theorem~\ref{thm:bipartitedfr}]
Let $p=p(A)$ be as in Lemma~\ref{lem:hypercontra} and $\lambda = \lambda(A)$ be the second absolute eigenvalue of $A$. 
Fix some $\eps>0$, and let $\tau_\moo(\eps)$ and $\delta_\moo(\eps)$ be as given in Lemma~\ref{lem:quantmoo}. 
Choose $\eta < 1$ close enough to $1$ so that $\eta > 1-\lambda$, $2 \sqrt{1-\eta} \le \delta_\moo(\eps) \eps / 2$,
and Eq.~\eqref{eq:noisyperturb} holds. 
Let $c>0$ be large enough so that $\delta := \eps^c < \sqrt{1-\eta}$. 
Define
\[
\ell = \frac{2 (1-\eta^2)^{-2}}{\tau_\moo(\eps)}\; ,
\]
and choose $\gamma > 0$ small enough so that 
\[
 2 \gamma < \tau_\moo(\eps) \cdot (\eps/2 \ell^2)^{2p/(p-2)}\; .
\]

For $i \in \{1,2\}$, define $g_i = N_\eta f_i$. Let $j$ be the number of variables 
with influence greater than $\gamma$ on either $g_1$ or $g_2$, and assume 
without loss of generality that these are the variables $J:=\{1,\ldots,j\}$. 
By Claim~\ref{clm:influencenoisy} there are at most $2(1-\eta^2)^{-2}/\gamma$ 
such variables, so in particular, we can take $j = \eps^{-c}$ for large enough $c$, as required. 

For $a \in V^j$ define $g_{1,a}:V^{n-j} \to [0,1]$ by $g_{1,a}(x)=g_1(a,x)$ and similarly for $g_2$.
Let
\[
T_1 = \{ a \,:\, \E_x[g_{1,a}(x)] \ge \eps \} \subseteq V^j \; ,
\]
and similarly define $T_2$ with $g_{2,a}$. The condition in Eq.~\eqref{eq:outsidetlowexpect} now follows
from Lemma~\ref{lem:dfrclmonnoiseoutside}.

It remains to prove that $\ip{\indic_{T_1}, A\indic_{T_2}} \le \eps$. 
Assume towards contradiction that 
$\ip{\indic_{T_1}, A\indic_{T_2}} > \eps$.
Equivalently, the measure of pairs $(a,b)$ such that $a \in T_1$ and $b \in T_2$ is greater than $\eps$.
Notice that by Lemma~\ref{lem:perturbnoiseip}, 
\[
 \ip{g_1, A g_2} \le \ip{f_1, A f_2} + \sqrt{1-\eta} \le \delta + \sqrt{1-\eta} \le 2\sqrt{1-\eta} \; .
\]
Therefore, the measure of pairs $(a,b)$ for which $\ip{g_{1,a},A g_{2,b}} \ge \delta_\moo(\eps)$
is at most $2 \sqrt{1-\eta} / \delta_\moo(\eps) \le \eps/2$. It follows that there is at least
an $\eps/2$ measure of pairs $(a,b)$ for which $\E_x[g_{1,a}(x)] \ge \eps$, $\E_x[g_{2,b}(x)] \ge \eps$,
and $\ip{g_{1,a},A g_{2,b}} < \delta_\moo(\eps)$. 
By Lemma~\ref{lem:quantmoo}, for each such pair $(a,b)$, there exists an $i \in \{j+1,\ldots,n\}$ whose
influence on both $g_{1,a}$ and $g_{2,b}$ is greater than $\tau_\moo(\eps)$. 
We can now apply Lemma~\ref{lem:hypercontra} with the sets
\[
   L(a) = \{ j < i \le n \,:\, \max(\inf_i(g_{1,a}),\inf_i(g_{2,a})) >  \tau_\moo(\eps) \} \; ,
\]
whose cardinality is at most $\ell$ by Claim~\ref{clm:influencenoisy},
and obtain that there exists an $i \in \{j+1,\ldots,n\}$ for which
\[
\mu(\{ a \in V^j \, : \, i \in L(a) \}) \ge (\eps/2\ell^2)^{2p/(p-2)} > 2 \gamma / \tau_\moo(\eps)\; .
\]
From this it follows that $i$ has influence greater than $\gamma$ on either
$g_1$ or on $g_2$, in contradiction to the definition of $J$. 
\end{proof}

%%% AUTHOR: optional acknowledgments here
\section*{Acknowledgments} %%  you may comment this out if no Ackno
We thank Michael Overton for developing HANSO and Pooya Hatami for sending us an early draft of their work~\cite{HatamiST13}. 

%%% AUTHOR:
%%% Bibliography goes here. Note that the arXiv cannot process bibtex
%%% or biber bibliographies.  Example of acceptable bibliograpy format:
%\bibliographystyle{amsplain}
%\bibliography{ashkara}

\providecommand{\bysame}{\leavevmode\hbox to3em{\hrulefill}\thinspace}
\providecommand{\MR}{\relax\ifhmode\unskip\space\fi MR }
\providecommand{\MRhref}[2]{%
  \href{http://www.ams.org/mathscinet-getitem?mr=#1}{#2}
}
\providecommand{\href}[2]{#2}

%% AUTHOR: You can generate such a bibliography from a .bib file by 
%% running pdflatex/bibtex/pdflatex/pdflatex and then pasting the .bbl file
%% between \begin{thebibliography} and \end{bibliography}

%%% AUTHOR: Include a short description of each author following the
%%% structure below. Use the same short tags used previously.  
%%% Use \imageat{} and \imagedot{} instead of "@" and "." in
%%% email addresses-this replaces the symbols with graphics to avoid 
%%% e-mail address harvesting from the .pdf file
\begin{dajauthors}
\begin{authorinfo}[friedgut]
  Ehud Friedgut\\
	Professor \\
  Faculty of Mathematics and Computer Science, Weizmann Institute of Science \\
	Rehovot, Israel\\
  %paulerdos\imageat{}renyiinstitute\imagedot{}hu \\
  \url{http://www.ma.huji.ac.il/~ehudf/}
\end{authorinfo}
\begin{authorinfo}[regev]
  Oded Regev\\
  Professor\\
  Courant Institute of Mathematical Sciences, New York University\\
  New York, NY, USA\\
  %johanh\imageat{}ktth\imagedot{}se \\
  \url{http://www.cims.nyu.edu/~regev/}
\end{authorinfo}
\end{dajauthors}

\end{document}